\documentclass[11pt]{amsart}
\usepackage{amssymb}
\usepackage{amsthm}
\usepackage{euscript}

\usepackage{graphicx}
\usepackage{bbding}
\usepackage[misc]{ifsym}
 \usepackage[usenames,dvipsnames]{color}

\newcommand{\cB}{\EuScript{B}}

\theoremstyle{plain}
\newtheorem{lem}[subsection]{Lemma}

\newtheorem{thm}[subsection]{Theorem}

\newtheorem{cor}[subsection]{Corollary}

\begin{document}
\title[Generators of full  algebras of operators]{Fully non-zero matrices and generators of full  algebras of operators}

\thanks{${}^1$ Research supported in part by NSERC (Canada).}
\thanks{${}^2$ Research supported in part by National Natural Science Foundation of China (No.: 12471123).}

\author
	[L.W. Marcoux]{{Laurent W.~Marcoux${}^1$}}
\address
	{Department of Pure Mathematics\\
	University of Waterloo\\
	Waterloo, Ontario \\
	Canada  \ \ \ N2L 3G1}
\email{LWMarcoux@math.uwaterloo.ca}


\author
	[H. Radjavi]{{Heydar Radjavi}}
\address
	{Department of Pure Mathematics\\
	University of Waterloo\\
	Waterloo, Ontario \\
	Canada  \ \ \ N2L 3G1}
\email{hradjavi@uwaterloo.ca}


\author
	[B.R. Yahaghi]{{Bamdad R. Yahaghi}}
\address
	{Department of Mathematics, Faculty of Sciences\\
	Golestan University\\
	Gorgan 19395-5746\\
	 Iran}
\email{bamdad5@hotmail.com, bbaammddaadd55@gmail.com}

\author
	[Y.H.~Zhang]{{Yuanhang~Zhang${}^2$}}
\address
	{School of Mathematics\\
	Jilin University\\
	Changchun 130012\\
	P.R. CHINA}
\email{zhangyuanhang@jlu.edu.cn}



\bibliographystyle{plain}

\begin{abstract}
In a wide variety of fields,  every non-scalar matrix $M$ is  similar to a matrix all of whose entries are non-zero.   We give extensions of this result, including some in the division ring context.   We then apply these results to the question of which pairs of operators generate the algebra of all operators on finite- and infinite-dimensional spaces.   In the case of complex, separable infinite-dimensional Hilbert space, this is intimately related to the unsolved Invariant Subspace Problem. Finally, linear independence of subsets of arbitrary TVS's and generators of arbitrary TA's are shown to be stable in the sense of the theorems presented.
\end{abstract}



\maketitle
\markboth{\textsc{  }}{\textsc{}}

\bigskip

\begin{center} \emph{To the memory of our good friends and colleagues, Eric Nordgren and Peter Rosenthal}
\end{center}

\bigskip

\section{Introduction} \label{sec01}

\bigskip

One topic of discussion in this paper is that of ``\emph{fully non-zero}" matrices.   Extending some known results, we show that under very general conditions on the underlying field (or suitable division ring), every non-scalar operator has a matrix (in some basis) whose entries are all non-zero.

The infinite-dimensional analogue of this result is known for (bounded linear) operators on complex, separable Hilbert space~\cite{RadjaviRosenthal1970}. We extend it  to the quaternionic setting.

The second topic, which as will be seen is closely related to the first,  is that of a small number of generators for an algebra of operators.   Of special interest is the question of when a pair of operators of very simple structure on a given linear space $\mathcal{V}$ generates the algebra of all operators on $\mathcal{V}$.  If $\mathcal{V}$ is an infinite-dimensional complex Hilbert space $\mathcal{H}$, which is of particular interest, the statement that ``$A$ and $B$ generate all operators" means, of course, that the weakly closed algebra generated by $A$ and $B$ is the algebra $\mathcal{B}(\mathcal{H})$ of all bounded linear operators on $\mathcal{H}$.

By ``simple structure" mentioned above we mean, for example, quadratic or cubic operators; i.e. operators satisfying equations $A^2 + \alpha A + \beta I = 0$, or $A^3 + \alpha A^2 + \beta A + \gamma I = 0$.   It is known~\cite{GLS} that two quadratic operators cannot generate the algebra of all $n \times n$ complex matrices if $n \ge 3$.  This makes it interesting that one quadratic and one cubic can do the job, even in infinite dimensions.

It is worthwhile to mention the intimate relationship between this generation question and the well-known and still unsolved Invariant Subspace Problem:  does every bounded linear operator on a (separable, complex, infinite-dimensional) Hilbert space $\mathcal{H}$ have an invariant (closed) subspace other than the trivial ones $\{ 0\}$ and $\mathcal{H}$?

This stubborn problem is equivalent~\cite{NordgrenRadjabalipourRadjaviRosenthal1988} to the easier-sounding question:  does every pair of quadratic operators on $\mathcal{H}$, for example $A$ and $B$ satisfying
\[
A^2 = I = B^2, \]
have a common non-trivial invariant subspace?  Without loss of generality, one of them can be assumed to be self-adjoint.   If $A=A^*$ and $B=B^*$, then the pair $\{ A, B\}$ has an abundance of invariant, actually \emph{reducing} subspaces in common~\cite[Theorem 1]{GZ}.   This implies that the pair is very far from being a generator for the algebra $\mathcal{B}(\mathcal{H})$.   But if there exists  a generating pair $\{ A, B \}$ of quadratic operators,  then the Invariant Subspace Problem is solved in the negative~\cite{NordgrenRadjabalipourRadjaviRosenthal1988}.

It is interesting that if a very slight deviation from quadraticity is allowed, we still get an affirmative result in the self-adjoint case; i.e., as a routine consequence of Theorem~\ref{thmJune9}, one may exhibit a pair $\{ A, B \}$ of self-adjoint generators for $\mathcal{B}(\mathcal{H})$, where $A^2 = I = B^3$, and $(B+F)^2=I$ for
some rank-one operator $F$.

In the finite-dimensional setting, we characterise all ``quadratic-cubic" pairs of self-adjoint generators for the full algebra $\mathbb{M}_n(\mathbb{C})$.   We also present stability results for generators of more general algebras.

\bigskip

\section{Preliminary results} \label{sec02}

\bigskip

A matrix $ A \in \mathbb{M}_n(F)$ is said to be {\it fully nonzero} if its entries are all nonzero; it is called a {\it fully nonzero matrix up to a similarity} if
$ P^{-1} A P $ is fully nonzero for some $ P \in \mathbb{M}_n(F)^{-1}$.


The following useful lemma is well-known. We provide an alternative proof, which is motivated by \cite[Lemma 3.1.1]{RadjaviRosenthal2000}, from the operator theory point of view.

\begin{lem} \label{L:subfield}
Let  $F$ be a field with characteristic zero and $\{ 0 \} \ne S \subseteq F$ a non-empty, countable subset of $F$. Then,   the subfield generated by  $S $ is isomorphic to a subfield of the complex numbers.
\end{lem}

\begin{proof}
Enumerate $S\cup\{1\}=\{s_j\}_{j=0}^N$, where $N\in \mathbb{N}\cup\{\infty\}$, and $s_0=1$. Without loss of generality,
we may assume that $S\cup\{1\}$ is linearly independent over $\mathbb{Q}$.

For $k\leq N$, let $F_k$ be the subfield of $F$ generated by $\{s_j\}_{j=0}^k$; here
$k< \infty$ simply means $ k\in \mathbb{N}$. Observe that $F_0\cong \mathbb{Q}$. We therefore
obtain one of the following two nests of field extensions, namely:
$$\mathbb{Q}\cong F_0 \leq F_1 \leq F_2 \leq \cdots \leq F_N, $$
or
$$ \mathbb{Q}\cong F_0 \leq F_1 \leq F_2 \leq \cdots \leq F_k \leq \cdots, $$
depending on whether or not $ N \in  \mathbb{N}$.
Now $F_k=F_{k-1}(s_k)$, $k\geq 1$. We argue by induction on $k$ that each $F_k$
is isomorphic to a subfield $G_k$ of $\mathbb{C}$.

Indeed, setting $G_0=\mathbb{Q}$, we have $G_0\cong F_0$.
For $k\geq 1$, if $s_k$ is algebraic over $F_{k-1}$, then
$F_k\cong G_k$, where $G_k$ is the subfield of $\mathbb{C}$
obtained by adjoining to $G_{k-1}$ a complex root $r_k$ of the
minimal polynomial for $s_k$ over $F_{k-1}$. If $s_k$ is transcendental over $F_{k-1}$, then
$F_k\cong G_k$, where $G_k$ is obtained from $G_{k-1}$ by adjoining any complex, transcendental number algebraically independent of $G_k$.  Set $G = \cup_{k=0}^N G_k$.

Thus,
 $$\mathbb{Q}= G_0 \leq G_1 \leq G_2 \leq \cdots \leq G_N\subseteq \mathbb{C}, $$
or
$$ \mathbb{Q}=G_0 \leq G_1 \leq G_2 \leq \cdots \leq G_k \leq \cdots\subseteq \mathbb{C}, $$
where $G_k\cong F_k$ for all $k$. It follows that
\[F\cong\bigcup_{k=0}^N G_k = G\subset \mathbb{C}.\]
\end{proof}

We also need the following simple lemma.

\begin{lem} \label{L:sum of nonzero elements}
Let  $D$ be a division ring and $2\leq  n \in \mathbb{N}$.
\begin{enumerate}
\item[(a)] If $D$ contains at least $n+2$ elements, then every nonzero element of $D$ can be written as the sum of $n$ distinct nonzero elements of $D$.
\item[(b)] If the characteristic of $D$ is not equal to $2$ and $D$ has at least $2n$ elements, then we may also express $0$ as a sum of $n$
distinct nonzero elements of $D$.
\item[(c)]  If the characteristic of $D$ is equal to $2$, $n\geq 3$ and $D$ has at least $2n$ elements, then we may also express $0$ as a sum of $n$
distinct nonzero elements of $D$.
\end{enumerate}

\end{lem}

\begin{proof}
\begin{enumerate}
\item[(a)] Suppose that $D$ is a division ring with at least $n+2$ elements. Pick $n$ distinct nonzero elements of $D$, say, $ x_1, \ldots, x_n$.
If $ x_1 + \cdots + x_n= 0$,
choose $y\in D\setminus \{0,x_1,x_2,\cdots,x_n\}$; otherwise, set $y=x_n$.
Either way, $\{x_1,x_2,\cdots, x_{n-1},y\}$ has cardinality $n$ and
\[x:=x_1+x_2+\cdots+x_{n-1}+y\neq 0.\]
Given $0\neq z\in D$,
\[z=(zx^{-1})x=\sum_{j=1}^{n-1}(zx^{-1})x_j+(zx^{-1})y,\]
and thus is a sum of $n$ distinct nonzero elements of $D$.

\item[(b)]

Suppose now that  $D$ has at least $2n$ elements, and that the characteristic of $D$ is not equal to 2.
If $n=2$, then as $1\neq -1$, $0$ is the sum of $1$ and $-1$.
Hence, we may assume that $n\geq 3$.
By part (a) above, we can find $n-2$
distinct nonzero elements  $ x_1 , \ldots, x_{n-2}$ such that $z_0:= x_1 + \cdots + x_{n-2} \neq 0$.

For $1\leq k\leq n-2$, set
$z_k:=x_k+z_0$.
Then the set
\[E:=\{0, x_1,x_2,\cdots,x_{n-2},-z_0,-z_1,\cdots,-z_{n-2},-(2)^{-1}z_0\}\]
has cardinality at most $2n-1$, and so we may choose
\[x_{n-1}\in D\setminus E.\]
Clearly, $x_1,x_2,\cdots, x_{n-1}$ are $n-1$ distinct nonzero elements of $D$, and for $1\leq k\leq n-2$,
\[x_k+(x_1+x_2+\cdots+x_{n-1})=z_k+x_{n-1}\neq 0,\]
while
\[x_{n-1}+(x_1+x_2+\cdots+x_{n-1})=z_0+2x_{n-1}\neq 0.\]
Set $x_n:=-(\sum_{j=1}^{n-1}x_j)=-(z_0+x_{n-1})\neq 0$.
If $x_n=x_k$ for some $1\leq k\leq n-1$,
then $x_k+\sum_{j=1}^{n-1}x_j=0$,
a contradiction as $z_k+x_{n-1}\neq 0$.
Thus, $x_1,x_2,\cdots,x_n$ are $n$ distinct nonzero elements of $D$ whose sum is $0$.

\item[(c)]
The proof in this case is very similar to that of (b).
As before, we can find $n-2$
distinct nonzero elements  $ x_1 , \ldots, x_{n-2}$ such that $z_0:= x_1 + \cdots + x_{n-2} \neq 0$.
Again, we set $z_k:=x_k+z_0$, $1\leq k\leq n-2$.
We define
\[E':=\{0, x_1,x_2,\cdots,x_{n-2},-z_0,-z_1,\cdots,-z_{n-2}\}.\]
Then $E'$
has cardinality at most $2n-2$, so we can find
\[x_{n-1}\in D\setminus E'.\]
Thus, $x_1,x_2,\cdots, x_{n-1}$ are distinct nonzero elements of $D$, and for $1\leq k\leq n-2$,
\[x_k+(x_1+x_2+\cdots+x_{n-1})=z_k+x_{n-1}\neq 0,\]
while
\[x_{n-1}+(x_1+x_2+\cdots+x_{n-1})=z_0\neq 0.\]
Setting $x_n:=-(\sum_{j=1}^{n-1}x_j)=-(z_0+x_{n-1})\neq 0$, we see by the same
argument as in (b) that
$x_1,x_2,\cdots,x_n$ are distinct nonzero elements of $D$ and
$x_1+x_2+\cdots+x_n=0$, completing the proof.
\end{enumerate}
\end{proof}

\noindent {\bf Remarks.}    1.  It follows from above that if the characteristic of $D$ is not $2$ and  $D$ contains at least $4$ elements, then every element of $D$ can be written as the sum of $2$ mutually distinct nonzero elements of $D$.

2. In contrast, if the characteristic of $D$ is 2, then the zero element of $D$ cannot be written as the sum of $2$ mutually distinct nonzero elements of it.

3. Indeed, in light of \cite[Proposition 1.1(v)]{Yahaghi2025}, the preceding lemma along with the above remarks  holds for alternative division algebras, and hence alternative division rings,  whose characteristics are not $2$.

\bigskip


\section{General results} \label{sec03}

\bigskip

The following is an algebraic counterpart of \cite[Theorem 2]{RadjaviRosenthal1970} for matrices with entries in a general field.


\begin{thm} \label{2.1}

{\rm (i)} Let $ n \in \mathbb{N}$ with $n > 2$, $F$ be any field with at least $ 4n -2$ elements, and $A \in \mathbb{M}_n(F)$ a nonscalar matrix. Then, there exists a $P \in  \mathbb{M}_n(F)^{-1}$  such that the entries of the matrix $ P^{-1} A P$ are all nonzero and its diagonal entries are mutually distinct.

{\rm (ii)} Let $F$ be a field with at least $6$ elements and $A \in \mathbb{M}_2(F)$ a nonscalar matrix. Then, there exists a
$P \in \mathbb{M}_2(F)^{-1}$  such that the entries of the matrix $ P^{-1} A P$ are all nonzero. Unless the characteristic of $F$ is $2$, we can also demand that the diagonal entries of  $ P^{-1} A P$ be mutually distinct.

\end{thm}

\bigskip

\begin{proof}  (i)  Since  $F$ has more than $2n  $ elements, by  \cite[Theorem 2]{Fillmore1969.02} and the Remark at the end of that paper, together with Lemma \ref{L:sum of nonzero elements} above, there  exists a $P \in  \mathbb{M}_n(F)^{-1}$  such that the  diagonal  entries of the matrix $ P^{-1} A P$ are all nonzero and mutually distinct. Pick a $P \in  \mathbb{M}_n(F)^{-1}$ such that the  diagonal  entries of the matrix $ P^{-1} A P$ are all nonzero and mutually distinct and that the number of its nonzero off diagonal entries is maximal. We claim that $P \in \mathbb{M}_n(F)^{-1}$ is a desired matrix. Suppose, by way of contradiction, not. If necessary applying similarities coming from permutation matrices, we may assume that the $(1,2)$-entry of $ Q := P^{-1} A P$ is zero. Let
$$ Q = \left(
\begin{array}{cc}
 Q_{11} & Q_{12} \\
 Q_{21} & Q_{22}
\end{array} \right), $$
where $ 0 \not= Q_{11} = \left(\begin{array}{cc}
 a & 0 \\
 b & c
\end{array} \right) \in \mathbb{M}_{2} ( F) $, $ Q_{12}$,  $ Q_{21} $, and $ 0 \not= Q_{22} $ are suitable block matrices.
Let
$$R_0:=  \left(\begin{array}{cc}
 1 & z \\
 0 & 1
\end{array} \right), \  R := R_0 \oplus I_{n-2} \in  \mathbb{M}_n(F), $$
where $ z \in F$.  We have
$$ R^{-1} Q R =\left(\begin{array}{cc}
R_0^{-1} Q_{11} R_0 & R_0^{-1} Q_{12} \\
 Q_{21} R_0 & Q_{22}
\end{array} \right),  $$
where
$$  R_0^{-1} Q_{11} R_0 = \left(\begin{array}{cc}
 a - zb & - b z^2 + (a- c)z\\
 b & bz + c
\end{array} \right) \in \mathbb{M}_{2\times 2} (F),$$
$$  R_0^{-1} Q_{12} =\left(\begin{array}{cc}
 \alpha_1 - z\alpha_2  \\
 \alpha_2
\end{array} \right)    \in \mathbb{M}_{2\times (n-2)} (F), $$
and
$$ \ Q_{21} R_0= \left(\begin{array}{cc}
 \beta_1  &   \beta_1 z + \beta_2\\
\end{array} \right)  \in \mathbb{M}_{ (n-2) \times 2} (F); $$
here, for $i = 1, 2$,  $\alpha_i$ (resp. $\beta_i$) stands for the $i$th row (resp. column) of the matrix $ Q_{12}$ (resp. $Q_{21}$).
Now, since  $F$ has more than $4n -3= 4(n-2) + 5$ elements,  we get that for an appropriate  $0 \not= z \in F$,  the $(1,2)$-entry of  the matrix  $ (PR)^{-1} A (PR)= R^{-1} Q R$  along with those entries whose corresponding ones in $ Q = P^{-1} A P$  are nonzero  are all nonzero and also its diagonal entries are mutually distinct. This is in contradiction with our choice of $ Q = P^{-1} A P$. This completes the proof.

(ii) The proof, which is omitted for brevity, is almost identical to that of (i).
\end{proof}


\noindent {\bf Remark.}
In the theorem, if we require only that the given nonscalar matrix be similar to a fully nonzero matrix, then it suffices for the ground field to have at least  $2n + 1$  elements. The proof goes through almost verbatim.

\bigskip


The following is the counterpart of \cite[Theorem 2]{RadjaviRosenthal1970} for a countable number of non-realscalar bounded linear operators on quaternion Hilbert spaces.


\begin{thm} \label{2.2}

{\rm (i)} Let $\mathcal{H}$ be a separable right (resp. left) quaternion Hilbert space and  $A_k \in \mathcal{B}( \mathcal{H}) $ ($k \in \mathbb{N}$) be non-realscalar bounded linear operators. Then, there exists an orthonormal basis $ \mathcal{B}$ for $\mathcal{H}$ such that $ \langle A_ke , e' \rangle \not= 0$ for all $ e, e' \in  \mathcal{B}$ and $k \in \mathbb{N}$.

{\rm (ii)} Let $ n \in \mathbb{N}$, $A_k \in \mathbb{M}_n(\mathbb{H})$ ($k \in \mathbb{N}$) be non-realscalar matrices. Then, there exists a unitary matrix $P \in  \mathbb{M}_n(\mathbb{H})$  such that the entries of $ P^* A_k P$'s are all nonzero for all $k \in \mathbb{N}$.
\end{thm}

\begin{proof} It suffices to prove (i), whose proof is carried out by adjusting those of \cite[Theorems 1 and 2]{RadjaviRosenthal1970} in the setting of quaternion Hilbert spaces.
\end{proof}



\begin{thm} \label{2.3}
Let  $ n \in \mathbb{N}$, $A \in \mathbb{M}_n(\mathbb{H})$ be such that its entries are all nonzero, and $ B \in  \mathbb{M}_n(\mathbb{H})$ a diagonal matrix whose diagonal entries are all nonzero and mutually nonsimilar. Then, up to a similarity, the real algebra generated by $A$ and $B$ is one of the following:  $ \mathbb{M}_n(\mathbb{R})$, $ \mathbb{M}_n(\mathbb{C})$, or  $\mathbb{M}_n(\mathbb{H})$.
\end{thm}

\begin{proof} In light of \cite[Theorem 1.2(iii)]{Yahaghi2017}, it suffices to show that the real algebra $\mathbb{A}$ generated by $A$ and $B$ is irreducible. To this end, if $n=1$, the assertion is easily verified. Suppose that  $n > 1$  and $\mathcal{M}$  is a nonzero invariant subspace of $\mathbb{A}$. We need to show that $ \mathcal{M} = \mathbb{H}^n$. Let $(e_k)_{k =1}^n $ be the standard ordered basis of $\mathbb{H}^n$ and $E_{ij}$ be the matrix with $1$ in its $(i,j)$-entry and $0$ elsewhere. Since $B$ is a diagonal matrix and its diagonal entries are all nonzero and mutually nonsimilar,  we get that $E_{ii} \in \mathbb{A}$ for all $ 1 \leq i \leq n$. From this, we see that $ a_{ij} E_{ij} = E_{ii} A E_{jj} \in \mathbb{A}$ for all $ 1 \leq i , j \leq n$.  Pick a nonzero $ x \in \mathcal{M}$ and let $x = e_1 x_1 + \cdots + e_n x_n$ for some $ x_i \in \mathbb{H}$ ($1 \leq i \leq n$). There exists a $ 1 \leq j \leq n$ such that $ x_j \not= 0$. But $ E_{jj} \in \mathbb{A}$, the subspace $\mathcal{M}$ is invariant under $\mathbb{A}$, and $ x \in \mathcal{M}$.  Thus, $ e_j x_j =  E_{jj} x \in  \mathcal{M}$, which yields $ e_j \in \mathcal{M}$. Once again,  $ a_{ij} E_{ij} = E_{ii} A E_{jj} \in \mathbb{A}$ for all $ 1 \leq i  \leq n$,  $\mathcal{M}$ is invariant under $\mathbb{A}$, and  $ e_j \in \mathcal{M}$. It thus follows that $ e_i a_{ij} =  a_{ij} E_{ij} (e_j) \in \mathcal{M}$, from which we obtain $ e_i \in \mathcal{M}$ for all $ 1 \leq i  \leq n$. This implies $ \mathcal{M} = \mathbb{H}^n$, as desired.
\end{proof}


\noindent {\bf Remark.} Adjusting the proof of the theorem, one can indeed prove the following. {\it Let  $ n \in \mathbb{N}$, $F$ a field, $ \Delta$ a finite-dimensional division $F$-algebra,  $A \in \mathbb{M}_n( \Delta)$ be such that its entries are all nonzero, and $ B \in  \mathbb{M}_n( \Delta)$
  a diagonal matrix whose diagonal entries are all nonzero and mutually nonsimilar. Then, the $F$-algebra generated by $A$ and $B$ is irreducible.}


The following is an immediate consequence of Theorems \ref{2.1} and \ref{2.3}.

\begin{cor} \label{2.4}
{\rm (i)} Let $ n \in \mathbb{N}$, $F$ be any field with at least $ 2n +1$ elements, and $A \in \mathbb{M}_n(F)$  a nonscalar matrix. Then, there exists a diagonalisable matrix $B \in \mathbb{M}_n(F)$ such that the algebra generated by $A$ and $B$ is $ \mathbb{M}_n(F)$.

{\rm (ii)} Let $ n \in \mathbb{N}$ and $A \in \mathbb{M}_n(\mathbb{H})$  be a non-realscalar matrix. Then, there exists a diagonalisable matrix $B \in \mathbb{M}_n( \mathbb{H})$ such that the real algebra generated by $A$ and $B$ is irreducible, and hence, up to a similarity, it is $ \mathbb{M}_n(\mathbb{R})$, $ \mathbb{M}_n(\mathbb{C})$, or  $\mathbb{M}_n(\mathbb{H})$.
\end{cor}

\begin{proof} The proof, which is omitted for brevity, follows quickly from the remark after Theorem \ref{2.1} and from Theorem  \ref{2.3} together with its proof.
\end{proof}

\bigskip


\section{Generators satisfying polynomials of minimum degrees} \label{sec04}


\bigskip

We now turn our attention to pairs of  ``\emph{small generators}" of $\mathcal{B}(\mathcal{H})$ for a finite-dimensional or an infinite-dimensional, separable complex Hilbert space $\mathcal{H}$.

As noted in the introduction, if $\mathcal{H}$ is of finite dimension greater than two, then $\mathcal{B}(\mathcal{H})$
does not admit a generating pair $\{A,B\}$ of quadratic operators. That being said, there exist unitary operators $U$ and $V$ satisfying $U^2 = V^3 = I$ that together generate $\mathcal{B}(\mathcal{H})$ as an algebra.   In fact, we shall  give a necessary and sufficient condition for a quadratic unitary and a cubic unitary to generate the whole algebra.

Let $n \in \mathbb{N}$, and suppose that $\mathcal{X}$, $\mathcal{Y}$ are subspaces of $\mathbb{M}_n(\mathbb{C})$.  We define
\[\textsc{MINT}_{U}(\mathcal{X},\mathcal{Y}) = \max \{ \dim \, (\mathcal{X} + U^{*} \mathcal{Y} U) : U \in \mathbb{M}_n(\mathbb{C}) \mbox{ unitary}\}.\]

The notation ``MINT" refers to ``minimal intersection", and this definition first arose in a rather different context in \cite{MarcouxRadjaviZhang2022}, in connection with the problem of determining
$\max\{\dim(P^\perp(\mathbb{A}\oplus \mathbb{B})P): P=P^*=P^2\}$, where $\mathbb{A}\subset \mathbb{M}_k(\mathbb{C})$ and $\mathbb{B}\subset \mathbb{M}_l(\mathbb{C})$ are semi-simple
algebras.

Given $T\in \mathcal{B}(\mathcal{H})$, we denote the unitary orbit of $T$ by
\[\mathcal{U}(T):=\{W^*TW: W\in \mathcal{B}(\mathcal{H})\mbox{ unitary}\}.\]

\begin{thm} \label{thm4.01}
Let $n\ge 3$, $\mathcal{H}=\mathbb{C}^{n}$, and suppose that $U$ and $V$ are unitary operators such that the minimal polynomial of $U$ (resp.~$V$) is a quadratic (resp.~a cubic) polynomial. Then there exist $U_{0}\in\mathcal{U}(U)$ and $V_{0}\in\mathcal{U}(V)$ such that the algebra $\mathfrak{M}$ generated by $U_{0}$ and $V_{0}$ is all of $\mathcal{B}(\mathcal{H})$ if and only if
\[
q_{1}+p_{1}\le n, \]
where $q_{1}$ is the larger of the multiplicities of the eigenvalues of $U$, and $p_{1}$ is the largest of the  multiplicities of the eigenvalues of $V$.
\end{thm}

\begin{proof}
Set $U \simeq \alpha_1 I_{q_1} \oplus \alpha_2 I_{q_2}$, where $q_1 \ge q_2$ and $\alpha_1 \ne \alpha_2$, and set $V \simeq \beta_1 I_{p_1} \oplus \beta_2 I_{p_2} \oplus \beta_3 I_{p_3}$, where $p_1 \ge p_2 \ge p_3$, and $\beta_1, \beta_2$ and $\beta_3$ are distinct.

\medskip

\noindent \textsc{The proof of necessity.}\\
Suppose that the condition on $p_{1}+q_{1}$ fails, i.e.\ $p_{1}+q_{1}>n$. Then, from the dimension formula it is easy to see that for any $U_{1}\in\mathcal{U}(U)$ and $V_{1}\in\mathcal{U}(V)$, $U_{1},V_{1}$ have a common eigenvector. Consequently
$$
\operatorname{Alg}(U_{1},V_{1})\neq\mathcal{B}(\mathcal{H}).
$$

\medskip
\noindent \textsc{The proof of sufficiency.}\\
Suppose now that $p_{1}+q_{1}\le n$. Since $q_{1}+q_{2}=n$ and $q_{1}\ge q_{2}$, it follows that $p_{1}\le n-q_{1}=q_{2}\le q_{1}$. By the functional calculus,
$$
U_{1}^{*}\in\operatorname{Alg}(U_{1}),\qquad V_{1}^{*}\in\operatorname{Alg}(V_{1})
$$
for any $U_{1}\in\mathcal{U}(U)$ and $V_{1}\in\mathcal{U}(V)$. Hence, for all such $U_1,V_1$,
$$
\operatorname{Alg}(U_{1},V_{1})=C^{*}(U_{1},V_{1}).
$$

%

\noindent Since $U\simeq  \alpha_{1}I_{q_{1}}\oplus \alpha_{2}I_{q_{2}}$, we deduce that for any $U_1\in \mathcal{U}(U)$,
$$
C^{*}(U_{1})' \cong M_{q_{1}}\oplus M_{q_{2}} =: \mathbb{B}.
$$

\noindent Similarly, since $V\simeq  \beta_{1}I_{p_{1}}\oplus \beta_{2}I_{p_{2}}\oplus \beta_{3}I_{p_{3}}$ we find that for any $V_1\in \mathcal{U}(V)$,
$$
C^{*}(V_{1})' \cong M_{p_{1}}\oplus M_{p_{2}}\oplus M_{p_{3}} =: \mathbb{A}.
$$

\noindent By \cite[Theorem 4.9]{MarcouxRadjaviZhang2022}, $\textsc{MINT}_{U}(\mathbb{A},\mathbb{B})=1$.   That is,
\[
\min\bigl\{\dim(\mathbb{A}\cap W^{*}\mathbb{B}W):\,W\in \mathcal{B}(\mathcal{H})=\mathbb{M}_n(\mathbb{C})\mbox{ unitary} \bigr\}=1. \]
Therefore there exist $U_{0}\in\mathcal{U}(U)$ and $V_{0}\in\mathcal{U}(V)$ such that
$$
C^{*}(U_0,V_0)'=C^{*}(U_{0})'\cap C^{*}(V_{0})' = \mathbb{C}I.
$$
Consequently,
$$
\operatorname{Alg}(U_{0},V_{0})=C^{*}(U_0,V_0) = \mathcal{B}(\mathcal{H}).
$$
\end{proof}


We next consider the case of separable, infinite-dimensional, complex Hilbert space $\mathcal{H}$.  We denote by $\textsc{wot}$ the weak-operator topology on $\mathcal{B}(\mathcal{H})$.


\begin{thm} \label{thmJune9}
Let $\mathcal{H}$ be an infinite-dimensional and separable Hilbert space, and suppose that $U$ and $V$ are unitary operators such that the minimal polynomial of $U$ (resp.~$V$) is a quadratic (resp.~a cubic) polynomial.
Then there exist $U_0 \in \mathcal{U}(U)$ and $V_0 \in \mathcal{U}(V)$ such that the  \textsc{wot}-closed algebra $\mathfrak{M}$ generated by $U_0$ and $V_0$  is all of $\mathcal{B}(\mathcal{H})$ if and only if
\begin{enumerate}
	\item[(a)]
	both of the eigenspaces for $U$ are infinite-dimensional, and
	\item[(b)]
	at least two of the eigenspaces of $V$ are infinite-dimensional.
\end{enumerate}
\end{thm}

\begin{proof}
It is routine to see that $U$  is diagonalisable with two  eigenvalues, say $\alpha_1$ and $\alpha_2$.   As such, the algebra generated by $U$ contains a diagonal operator $D_U$ with eigenvalues $1$ and $-1$ with multiplicities equal to the multiplicities of $\alpha_1$ and $\alpha_2$ respectively.   Similarly, the algebra generated by $V$ contains a diagonalisable operator $D_V$ with eigenvalues $1$, $\omega := e^{\frac{2 \pi i}{3}}$ and $\omega^2$ with multiplicities equal to the multiplicities of the eigenvalues $\beta_1, \beta_2$, and $\beta_3$ of $V$.    As such, if we can find unitary operators $W_1, W_2$ such that  the \textsc{wot}-closed algebra generated by $X_0 = W_1^* D_U W_1$ and $Y_0 = W_2^* D_V W_2$ is all of $\mathcal{B}(\mathcal{H})$, then clearly the \textsc{wot}-closed algebra generated by $W_1^* U W_1$ and $W_2^* V W_2$ is all of $\mathcal{B}(\mathcal{H})$, and we are done.  Conversely, we observe that $U$ lies in the algebra generated by $D_U$, and that $V$ lies in the algebra generated by $D_V$.  In other words, we have reduced the problem to the case where $U^2 = I$ and $V^3 = I$.

\medskip

\noindent{\textsc{the proof of necessity}}

Suppose that condition (a) fails, and that (without loss of generality), the eigenspace $\mathcal{H}(U, \{ -1\})$ for $U$ corresponding to the eigenvalue $-1$ is finite-dimensional.   If $U_0 \in \mathcal{U}(U)$ and $V_0 \in \mathcal{U}(V)$, then -- since at least one of the eigenspaces for $V_0$ is infinite-dimensional, call it $\mathcal{H}(V_0, \{ \beta \})$ (where $\beta \in \{ 1, \omega, \omega^2\}$), it follows that $\mathcal{H}(U_0, \{ 1\}) \cap \mathcal{H}(V_0, \{ \beta\}) \ne \{ 0\}$, and so if $0 \ne e \in  \mathcal{H}(U_0, \{ 1\}) \cap \mathcal{H}(V_0, \{ \beta\})$, then $\mathbb{C} e$ is reducing for both $U_0$ and $V_0$.     In particular, the \textsc{wot}-closure of the algebra generated by $U_0$ and $V_0$ is not all of $\mathcal{B}(\mathcal{H})$.

\medskip

Next, suppose that condition (b) fails,  say -- without loss of generality -- that the eigenvalues $\omega$ and $\omega^2$ appear with finite multiplicity.
Again, if $U_0 \in \mathcal{U}(U)$, then at least one of its eigenspaces is infinite-dimensional, say $\mathcal{H}(U_0, \{ \alpha\})$, where $\alpha \in \{-1, 1\}$, and so $\mathcal{H}(U_0, \{ \alpha \}) \cap \mathcal{H}(V_0, \{ 1 \}) \ne \{ 0\}$.   If $0 \ne f \in \mathcal{H}(U_0, \{ \alpha \}) \cap \mathcal{H}(V_0, \{ 1 \})$, then $\mathbb{C} f$ is reducing for both $U_0$ and $V_0$, and again, the \textsc{wot}-closure of the algebra generated by $U_0$ and $V_0$ is not all of $\mathcal{B}(\mathcal{H})$.

\medskip

\noindent{\textsc{the proof of sufficiency}}

To show sufficiency, write $V$ as a direct sum
\[
V = \begin{bmatrix} R & 0 \\ 0 & T \end{bmatrix}, \]
where the underlying spaces of $R$ and $T$ are both infinite-dimensional, $R$ has two eigenvalues and $T$ has just one.  Since $R$ is non-scalar, we can assume with no loss of generality that it has a fully nonzero matrix \cite[Theorem 2]{RadjaviRosenthal1970}. Again without loss of generality, assume that in this decomposition of the space, $U$ is the symmetry with the corresponding matrix
\[
U = \begin{bmatrix} M & N \\ N & - M \end{bmatrix}, \]
where $M$ is diagonal positive matrix with distinct diagonal entries all between $\frac{1}{3}$ and $\frac{2}{3}$.  Define $N > 0$ so that
\[
M^2   +   N^2 =  I. \]

By the functional calculus, the projection $P = \begin{bmatrix} I & 0 \\ 0 & 0 \end{bmatrix}$ lies in the algebra generated by $V$ (as does $P^\perp$), and thus the \textsc{wot}-closed algebra $\mathfrak{M}$ generated by $U$ and $V$ contains the \textsc{wot}-closed algebra $\mathfrak{M}_{11}$ generated by $P U P$ and $P V P$.   But $M$ is diagonal with distinct eigenvalues and $R$ is everywhere non-zero, whence $\mathfrak{M}_{11} = \cB(P \mathcal{H})$.

Note that $N$ is invertible.   Let $X_1 \in \cB(P^\perp \mathcal{H}, P \mathcal{H})$, $X_2 \in \cB(P \mathcal{H}, P^\perp \mathcal{H})$, and set $Y = X_1 N^{-1}$, $Z = N^{-1} X_2 \in \cB(P \mathcal{H})$.   Then
\[
\begin{bmatrix} 0 & X_1 \\ 0 & 0 \end{bmatrix} = \begin{bmatrix} Y & 0 \\ 0 & 0 \end{bmatrix} \begin{bmatrix} 0 & N \\ 0 & 0 \end{bmatrix} \in \mathfrak{M}, \]
and
\[
\begin{bmatrix} 0 & 0 \\ X_2 & 0 \end{bmatrix} = \begin{bmatrix} 0 & N \\ 0 & 0 \end{bmatrix}  \begin{bmatrix} Z & 0 \\ 0 & 0 \end{bmatrix} \in \mathfrak{M}. \]

It is then routine to check that $\mathfrak{M} = \mathcal{B}(\mathcal{H})$.
\end{proof}

\bigskip

\section{Stability of general sets of generators} \label{sec05}

\bigskip

In this section we give a general form of the following statement:   if $\mathcal{S} = \{ a_1, a_2, \ldots, a_k \}$ is a set of generators for an arbitrary algebra, and $\mathcal{T} = \{ b_1, b_2, \ldots, b_k\}$ is another set which is ``\emph{sufficiently close}" to $\mathcal{S}$ (in a sense to be made precise below), then $\mathcal{T}$ also generates the algebra.

In the following, the algebra $\mathbb{A}$ is arbitrary, that is, it is not even assumed to be associative. By an {\it algebra}, we mean a vector space  $\mathbb{A}$ over a field $F$ together with a multiplication coming from an $F$-bilinear function on  $\mathbb{A}$. With $ \mathbb{F} \in \{ \mathbb{R}, \mathbb{C}\}$, by a {\it topological  $ \mathbb{F}$-algebra}, we mean an $ \mathbb{F}$-algebra $\mathbb{A}$ endowed with a Hausdorff topology with respect to which the algebraic operations of $\mathbb{A}$, namely, the addition, the product of vectors by scalars, and the multiplication of $\mathbb{A}$, are all continuous. At this point, it is worth noting that, in light of  \cite[Theorem 1.21]{Rudin1991} and \cite[Theorem 1.1(i)]{Yahaghi2024}, any finite-dimensional $ \mathbb{F} $-algebra together with any vector space topology is a topological  $ \mathbb{F}$-algebra;  also see \cite[Theorem 4]{Albert1947}  and \cite[Proposition 1.1.7]{CabreraRodriguezPalacios2014}.

 The following theorem shows that  linear independence of a finite set of vectors in an arbitrary topological vector space is topologically stable in the following sense; see \cite{Fischer}.


\begin{thm} \label{5.1}
{\rm (i)} Let   $ \mathbb{F} \in \{ \mathbb{R}, \mathbb{C},  \mathbb{H}\}$, $ k \in \mathbb{N}$, and $X$ be a left (resp. right) topological vector space over $ \mathbb{F}$.  Then, the set
$$\mathcal{U}_k := \Big\{ (x_1 , \ldots , x_k ) \in X^k : \ \{ x_1 , \ldots , x_k\} \ {\rm is \ linearly \ independent} \Big\} $$
is open with respect to the product topology of $  X^k$. In particular, if  $(X, \|.\|)$ is a left (resp. right) normed linear space over $ \mathbb{F}$,  then, given linearly independent vectors $x_i \in X$ $( 1 \leq i \leq k)$, there exists an $ \varepsilon > 0$ such that $ \{ y_1 , \ldots, y_k \} \subseteq X$ is linearly independent  whenever
$$\max_{ 1 \leq i \leq k}  \| y_i - x_i\|  < \varepsilon.$$

{\rm (ii)} Under the hypotheses of (i), the set  $\mathcal{U}_k$ is dense in $  X^k$ whenever it is nonempty.

\end{thm}

\begin{proof}
(i) It suffices to prove the first assertion. To this end, we proceed by way of contradiction and suppose that there exists an $ x = (x_1 , \ldots , x_k ) \in \mathcal{U}_k $ which is not an interior point of $ \mathcal{U}_k $. It follows that there exist a directed set $J$ and nets
$(x_{ij})_{j \in J}$ with $ x_{ij} \in X$ ($ 1 \leq i \leq k$)  and
$ \lim_J x_{ij} = x_i $ ($ 1 \leq i \leq k$) such that $ \{ x_{1j}, \ldots , x_{kj} \}$ is linearly dependent for all
$ j \in J$. Indeed, a directed set $J$ is the set of all neighborhoods of $x$ in $X^k$ partially ordered by reverse inclusion. It follows that for each $ j \in J$,
there are $ a_{ij} \in \mathbb{F}$
($ 1 \leq i \leq k$) with $\max_{1 \leq i \leq k} |a_{ij} | = 1$ such that
$$ a_{1j} x_{1j} + \cdots +  a_{kj} x_{kj} = 0.$$
    Now, if necessary,  by renaming  $x_i$'s ($ 1 \leq i \leq k$) and passing to proper subnets,
we may assume that $\max_{1 \leq i \leq k} |a_{ij} | = 1= |a_{1j}|$ for all $ j \in J$ and that $ \lim_J a_{ij}  = a_i \in \mathbb{F}$ exists for all $ 1 \leq i \leq k$.
Thus, we get that
\begin{eqnarray*}
 	0 & = & \lim_{J} \big(a_{1j} x_{1j} + \cdots +  a_{kj} x_{kj}\big), \\
  	& = &  a_1 x_1 + \cdots + a_k x_k,
\end{eqnarray*}
implying that $ \{ x_1 , \ldots, x_k \}$ is linearly dependent, for $|a_1| = 1$, which is a contradiction. This completes the proof.

(ii) We prove the assertion by induction on $k$. The proof is easy if $ k=1$. Assuming that the assertion holds for $k$, we prove it for $ k+1$. To this end, let  $\mathcal{U}_{k+1} \not= \emptyset$  and  $ (x_1 , \ldots, x_{k+1}) \in   X^{k+1}$ be arbitrarily given. We need to show that every neighborhood of  $ (x_1 , \ldots, x_{k+1}) \in   X^{k+1}$ intersects  $\mathcal{U}_{k+1}$. In other words, given open sets  $G_i$ with $ x_i \in G_i$   ($1 \leq i \leq k+1$)  arbitrarily, there are $ y_i \in G_i$ ($1 \leq i \leq k+1$) such that $ (y_1, \ldots, y_{k+1}) \in \mathcal{U}_{k+1}$. Since  $\mathcal{U}_{k} \not= \emptyset$, by the inductive hypothesis,  there exist  $y_i \in G_i$   ($1 \leq i \leq k$)
 such that $  (y_1, \ldots, y_k) \in \mathcal{U}_k$. Note that $G_{k+1} \not\subseteq {\rm span}\{y_1, \ldots, y_k \} $, for otherwise  $ X = {\rm span}\{y_1, \ldots, y_k \}$ because $X$ is a topological vector space.  This, in turn,  yields $ k+1 \leq \dim X=  k$, which is impossible. Thus,  there exists a $ y_{k + 1} \in G_{k+ 1}$ such that $ (y_1 , \ldots, y_{k+1}) \in  \mathcal{U}_{k+1}$, as desired.
\end{proof}

\begin{thm} \label{5.2}
Let   $ \mathbb{F} \in \{ \mathbb{R}, \mathbb{C}\}$, $ k \in \mathbb{N}$,  and $\mathbb{A} $ be an arbitrary topological $\mathbb{F}$-algebra.  Suppose further that $A_i \in \mathbb{A}$ $( 1 \leq i \leq k)$ are such that  $ {\rm Alg}_\mathbb{F}  \big(\{  A_1 , \ldots, A_k \}\big) $ is finite-dimensional. Then,  there exist open sets $ \mathcal{U}_i \subseteq \mathbb{A} $ with  $A_i \in \mathcal{U}_i$ such that
 \[\dim {\rm Alg}_\mathbb{F}  \big(\{  X_1 , \ldots, X_k \} \big) \geq \dim {\rm Alg}_\mathbb{F}  \big(\{  A_1 , \ldots, A_k \}\big)\]
  whenever  $ X_i \in \mathcal{U}_i$ for all $ 1 \leq i \leq k$.

  In particular, if $(\mathbb{A}, \|.\|) $ is an arbitrary normed $\mathbb{F}$-algebra, then there exists an $ \varepsilon > 0$ such that
\[ \dim {\rm Alg}_\mathbb{F}  \big(\{  X_1 , \ldots, X_k \} \big) \geq \dim {\rm Alg}_\mathbb{F}  \big(\{  A_1 , \ldots, A_k \}\big)\]
 whenever
$$\max_{ 1 \leq i \leq k}  \| X_i - A_i\|  < \varepsilon.$$
\end{thm}

\begin{proof} Clearly, it suffices to prove the first assertion.   Let $$ d = \dim {\rm Alg}_\mathbb{F}  \big(\{  A_1 , \ldots, A_k \}\big).$$
It follows that  there are monic noncommutative monomials $ m_i$'s, respectively, of total degrees $n_i $'s  ($ 1 \leq i \leq d$) in the indeterminates $ x_1, \ldots , x_k$    such that
$$ \big\{  m_1( A_1 , \ldots, A_k) , \ldots,  m_d ( A_1 , \ldots, A_k) \big\} $$
is linearly independent. Here, naturally, $m_i( A_1 , \ldots, A_k)$ stands for the the monomial $m_i$ evaluated at $( A_1 , \ldots, A_k) \in \mathbb{A}^k $
 ($ 1 \leq i \leq d$).
By the preceding theorem, there exist open sets $ \mathcal{V}_i \subseteq \mathbb{A}$ with  $ m_i( A_1 , \ldots, A_k) \in \mathcal{V}_i $ ($ 1 \leq i \leq d$) such that $ \{ B_1 , \ldots, B_d\} \subseteq \mathbb{A}$ is linearly independent whenever  $( B_1 , \ldots, B_d)  \in \mathcal{V}_1 \times \cdots \times \mathcal{V}_d $.
Define $ m :  \mathbb{A}^k \longrightarrow  \mathbb{A}^d$ by
\[m(X) = ( m_1 (X) , \ldots, m_d(X) ),\] where $ X = (X_1 , \ldots, X_k)$. Equip  $\mathbb{A}^k$ and $\mathbb{A}^d$ with the product topologies. Since $\mathbb{A} $ is a topological algebra, a straightforward induction on the total degree of $m_i$  reveals that every monomial function $ m_i    :  \mathbb{A}^k \longrightarrow  \mathbb{A}$ ($ 1 \leq i \leq d$) is continuous. From this, we get that the function $m = (m_1 , \ldots, m_d)$ is continuous.  But $ \mathcal{V}_1 \times \cdots  \times \mathcal{V}_d $ is open in  $\mathbb{A}^d$, $ m_i(A) \in \mathcal{V}_i $ for each $ i =1 , \ldots, d$, and the function $ m :  \mathbb{A}^k \longrightarrow  \mathbb{A}^d$ is continuous at $A = ( A_1 , \ldots, A_k)$. It thus follows that there are open sets  $ \mathcal{U}_j $ in   $\mathbb{A}$ with $ A_j \in \mathcal{U}_j $
($1 \leq j \leq k$) such that $m( X_1 , \ldots, X_k) \in  \mathcal{V}_1 \times \cdots  \times \mathcal{V}_d  $  whenever $  ( X_1 , \ldots, X_k) \in \mathcal{U}_1 \times \cdots \times \mathcal{U}_k$.  Now, letting $ X = ( X_1 , \ldots, X_k) \in \mathbb{A}^k $, we see that
 $ \big\{   m_1 (X) , \ldots, m_d(X)  \big\} $ is linearly independent in $\mathbb{A}$ whenever $ X  \in \mathcal{U}_1 \times \cdots \times \mathcal{U}_k $ .  This, in turn, yields
$$ \dim {\rm Alg}_\mathbb{F}  \big(\{  X_1 , \ldots, X_k \} \big) \geq d = \dim {\rm Alg}_\mathbb{F}  \big(\{  A_1 , \ldots, A_k \}\big)$$
whenever  $  ( X_1 , \ldots, X_k) \in \mathcal{U}_1 \times \cdots \times \mathcal{U}_k$, completing the proof.
\end{proof}

In finite dimensions, we can assert a bit more.

\begin{thm} \label{5.3}
{\rm (i)}
Let   $ \mathbb{F} \in \{ \mathbb{R}, \mathbb{C}\}$, $ k \in \mathbb{N}$,  and $\mathbb{A} $ be an arbitrary  finite-dimensional  $\mathbb{F}$-algebra endowed with any vector space topology.  Then, given  $A_i \in \mathbb{A}$ $( 1 \leq i \leq k)$, there exist  open sets $ \mathcal{U}_i \subseteq \mathbb{A} $ with  $A_i \in \mathcal{U}_i$ such that $ \dim {\rm Alg}_\mathbb{F}  \big(\{  X_1 , \ldots, X_k \} \big) \geq \dim {\rm Alg}_\mathbb{F}  \big(\{  A_1 , \ldots, A_k \}\big)$  whenever  $ X_i \in \mathcal{U}_i$ for all $ 1 \leq i \leq k$. In particular, if ${\rm Alg}_\mathbb{F}  \big(\{  A_1 , \ldots, A_k \}\big) = \mathbb{A}$, then
\[
{\rm Alg}_\mathbb{F}  \big(\{  X_1 , \ldots, X_k \} \big)  = \mathbb{A} \]
whenever  $ X_i \in \mathcal{U}_i$ for all $ 1 \leq i \leq k$.

{\rm (ii)} Let   $ \mathbb{F} \in \{ \mathbb{R}, \mathbb{C}\}$, $ k \in \mathbb{N}$,  $\mathbb{A} $ be an arbitrary finite-dimensional  $\mathbb{F}$-algebra, and $ \| .\| $ any vector space norm on $\mathbb{A} $. Suppose that $A_i \in \mathbb{A}$ $( 1 \leq i \leq k)$ are such that ${\rm Alg}_\mathbb{F}  \big(\{  A_1 , \ldots, A_k \}\big)= \mathbb{A}$. Then, there exists an $ \varepsilon > 0$ such that ${\rm Alg}_\mathbb{F}  \big(\{  X_1 , \ldots, X_k \} \big)  = \mathbb{A}$ whenever
$$\max_{ 1 \leq i \leq k}  \| X_i - A_i\|  < \varepsilon.$$
\end{thm}

\begin{proof}It suffices to prove (i). As pointed out before Theorem \ref{5.1}, from  \cite[Theorem 1.21]{Rudin1991} and \cite[Theorem 1.1(i)]{Yahaghi2024} we get that any finite-dimensional $ \mathbb{F} $-algebra together with any vector space topology is indeed a topological  $ \mathbb{F}$-algebra. Now, the assertion immediately follows from Theorem \ref{5.2}.
\end{proof}


\begin{thebibliography}{999}


\vspace{1mm}
\bibitem{Albert1947}
A.A. Albert, Absolute valued real algebras, {\it Ann. of Math.} 48 (1947), 495-501.


\vspace{1mm}
\bibitem{GZ}
G.R. Allan and  J. Zemanek,
Invariant subspaces for pairs of projections.
J. London Math. Soc. (2) 57 (1998), no. 2, 449--468.

\vspace{1mm}
\bibitem{CabreraRodriguezPalacios2014}
M. Cabrera  and  A.   Rodr\'iguez, {\it Non-Associative Normed Algebras,  Vol. I: The Vidav-Palmer and
Gelfand-Naimark Theorems}, Cambridge University Press, Cambridge, 2014.

\vspace{1mm}
\bibitem{Fillmore1969.02}
P. Fillmore, On similarity and the diagonal of a matrix, Am. Math. Mon. 76, 167-169 (1969).


\vspace{1mm}
\bibitem{Fischer}
D. Fischer,
{\it https://math.stackexchange.com/questions/1685682/regarding-linear-independence-on-a-normed-linear-space-given-a-condition}, April 2017.


\vspace{1mm}
\bibitem{GLS}
F.J. Gaines, T.J. Laffey, and H.M. Shapiro,  Pairs of matrices with quadratic minimal polynomials. Linear Algebra Appl. 52/53 (1983), 289--292.



\vspace{1mm}
\bibitem{MarcouxRadjaviZhang2022}
L.W. Marcoux, H. Radjavi, and Y. Zhang,
 Dispersing representations of  semi-simple subalgebras of complex matrices,  Linear Algebra Appl. 642  (2022),
160-220.



\vspace{1mm}
\bibitem{NordgrenRadjabalipourRadjaviRosenthal1988}
E. Nordgren, M. Radjabalipour, H. Radjavi, and P. Rosenthal, Quadratic operators and invariant subspaces, Studia Math. 88 (1988) 263--268.



\vspace{1mm}
\bibitem{RadjaviRosenthal1970}
H. Radjavi and P. Rosenthal, Matrices for operators and generators of $B (\mathcal{H})$, J. London Math. Soc. (2), 2 (1970), 557-560.



\vspace{1mm}
\bibitem{RadjaviRosenthal2000}
H. Radjavi and P. Rosenthal, {\it Simultaneous Triangularization}, Universitext, Springer-Verlag, New York,  2000.


\vspace{1mm}
\bibitem{Rudin1991}
W. Rudin, {\it Functional Analysis}, 2nd edition, McGraw-Hill, Inc.,  New York, 1991.

\vspace{1mm}
\bibitem{Yahaghi2017}
B.R. Yahaghi, Burnside type theorems in real and quaternion settings, {\em arXiv: 1710.03849v2}, 2017.


\vspace{1mm}
\bibitem{Yahaghi2024}
B.R. Yahaghi, Extensions of the fundamental theorem of algebra,  {\em  arXiv:2203.14689v5}, 2024.

\vspace{1mm}
\bibitem{Yahaghi2025}
B.R. Yahaghi, Spectrum in alternative topological algebras and a new look at old theorems,  {\em  arXiv:2401.00985v6}, 2025.


\end{thebibliography}
\end{document}